\documentclass[12pt]{amsart}
\usepackage[utf8]{inputenc}
\usepackage{amsmath}
\usepackage{graphics}
\usepackage{amsthm, multicol}
\usepackage{bm}
\usepackage{amssymb}
\usepackage{setspace, tikz, float,graphicx, tikz-cd}
\usepackage{mathrsfs}
\usepackage{mathabx}
\usepackage{interval}
\usepackage{amsfonts,latexsym,amscd,euscript,graphicx}

\usepackage{framed, mathtools}
\usepackage{fullpage}
\usepackage{hyperref}
\usepackage[margin=0.8 in]{geometry}
\usepackage{setspace}
\usepackage{enumerate}
\newtheorem{thm}{Theorem}
\newtheorem{lemma}[thm]{Lemma}

   \hypersetup{colorlinks=true,citecolor=blue,urlcolor =black,linkbordercolor={1 0 0}}

\allowdisplaybreaks[1]
\usepackage{thmtools}
\declaretheoremstyle[notefont=\bfseries,notebraces={}{},%
    headpunct={},postheadspace=1em,spaceabove=0.5em,spacebelow=0.5em]{mystyle}
\declaretheorem[style=mystyle,numbered=no,name=Theorem]{thm-hand}
\declaretheorem[style=mystyle,numbered=no,name=Conjecture]{conj-hand}
\declaretheorem[style=mystyle,numbered=no,name=Definition.]{defn-hand}
\declaretheorem[style=mystyle,numbered=no,name=Theorem.]{thm-no}
\declaretheorem[style=mystyle,numbered=no,name=Conjecture.]{conj-no}
\declaretheorem[style=mystyle,numbered=no,name=Lemma.]{lemma-no}
\declaretheorem[style=mystyle,numbered=no,name=Lemma]{lemma-hand}
\makeatletter
\def\resetMathstrut@{%
  \setbox\z@\hbox{%
    \mathchardef\@tempa\mathcode`\[\relax
    \def\@tempb##1"##2##3{\the\textfont"##3\char"}%
    \expandafter\@tempb\meaning\@tempa \relax
  }%
  \ht\Mathstrutbox@\ht\z@ \dp\Mathstrutbox@\dp\z@}
\makeatother
\begingroup
  \catcode`(\active \xdef({\left\string(}
  \catcode`)\active \xdef){\right\string)}
  \endgroup
\mathcode`(="8000 \mathcode`)="8000

\newtheorem{question}[thm]{Question}
\newtheorem{corollary}[thm]{Corollary}
\newtheorem{conj}[thm]{Conjecture}

\newtheorem{notation}[thm]{Notation}
\theoremstyle{definition}
\newtheorem{defn}[thm]{Definition}

\usepackage{listings}
\theoremstyle{remark}

\newcommand{\Z}{\mathbb Z}

\newcommand{\x}{\mathbf x}

\newcommand{\z}{\mathbf z}

\newcommand{\y}{\mathbf y}

\newcommand{\nc}{\newcommand}

\nc{\on}{\operatorname}
\nc{\Spec}{\on{Spec}}
\vspace{-15mm}
\title{Hypercube Packings and Coverings with Higher Dimensional Rooks} 

\author{Mehtaab Sawhney}
\thanks{Massachusetts Institute of Technology, Cambridge MA. Email: \texttt{msawhney@mit.edu}}
\author{David Stoner}
\thanks{Harvard University, Cambridge MA. Email: \texttt{dstoner@college.harvard.ed}}
\date{\today}
\begin{document}
\begin{abstract}
We introduce a generalization of classical $q$-ary codes by allowing points to cover other points that are Hamming distance $1$ or $2$ in a freely chosen subset of all directions. More specifically, we generalize the notion of $1$-covering, $1$-packing, and $2$-packing in the case of $q$-ary codes. In the covering case, we establish the analog of the sphere-packing bound and in the packing case, we establish an analog of the singleton bound. Given these analogs, in the covering case we establish that the sphere-packing bound is asymptotically never tight except in trivial cases. This is in essence an analog of a seminal result of Rodemich regarding $q$-ary codes. In the packing case we establish for the $1$-packing and $2$-packing cases that the analog of the singleton bound is tight in several possible cases and conjecture that these bounds are optimal in general.
\end{abstract}
\maketitle
\section{Introduction}
Consider a set of $n$ football matches which each end in either a win, draw, or loss. How many bets are necessary for an individual to guarantee that they predict at least $n-1$ of outcomes of the games correctly? What about having at least $n-k$ outcomes correct?

The above problem is the classical Football Pool Problem that has been extremely well studied for small specific values of $n$, as well as the generalization allowing for ``more" possible outcomes \cite{blokhuis1984more,fernandes1983football,haas2007lower,haas2009lower,hamalainen1991upper, kalbfleisch1969combinatorial,koschnick1993new, rodemich1970coverings,singleton1964maximum}.
In particular, consider the hypercube $H_{n,k}$, the $k$-dimensional hypercube with side length $n$. Then place $H_{n,k}$ on the lattice $\{0,\ldots,n-1\}^k$ and define the distance between two points in $H_{n,k}$ to be the Hamming distance between their coordinate representations. In other words, the Hamming distance is the number of places in which the coordinate representations of the points differ. An $R$-\emph{covering} is a set of points $S$ such that every point in the hypercube is within distance $R$ of a point in $S$ \cite{van1991bounds}. In the literature the minimum possible size of an $R$-covering has been the primary subject of interest, especially when $R=1$. Similarly, a set $T$ is an $R'$-\emph{packing} if no two points are within distance $R'$ of each other. In this case the primary object of study is the largest possible $R'$-packing \cite{van1991bounds}. For the remainder of this paper, we will focus on generalizations of the well studied cases where $R=1$ and $R'=1,2$.

Given the extensive research in the case where points can cover in all directions parallel to the axes, we instead consider the generalization where each point can cover in only a subset of these directions. In particular define an $\ell$-\emph{rook} to be rook which can cover in $\ell$ dimensions. More precisely, an $\ell$-rook is a point in $\Z^k$ along with a selection of $\ell$ out of $k$ coordinates and this point covers exactly the points which differ in one of the $\ell$ chosen coordinates. For example a $2$-rook in two dimensional space is a regular planar rook. Given this close relation with the chessboard piece, we use the terms ``attack" and ``cover" interchangeably. With this notion, we can now define the primary objects of study for this paper.

\begin{defn}
Let $n$, $k$, $\ell$ be positive integers with $k\ge \ell$. Define $a_{n,k,\ell}$ to be the minimum number of $\ell$-rooks that can cover $H_{n,k}$. Similarly define $b_{n,k,\ell}$ to be the maximum number of $\ell$-rooks in $H_{n,k}$ with no rooks attacking another. Finally, define $c_{n,k,\ell}$ to be the maximum number of $\ell$-rooks that can be placed in $H_{n,k}$ so that no two rooks attack the same point. Note that in all of these cases we do not allow multiple rooks at a single point.
\end{defn}
Below are three figures which demonstrate the optimal constructions for $a_{n, k, \ell}, b_{n, k, \ell}$, and $c_{n, k, \ell}$ in the case $(n, k, \ell)=(3, 3, 2)$. 

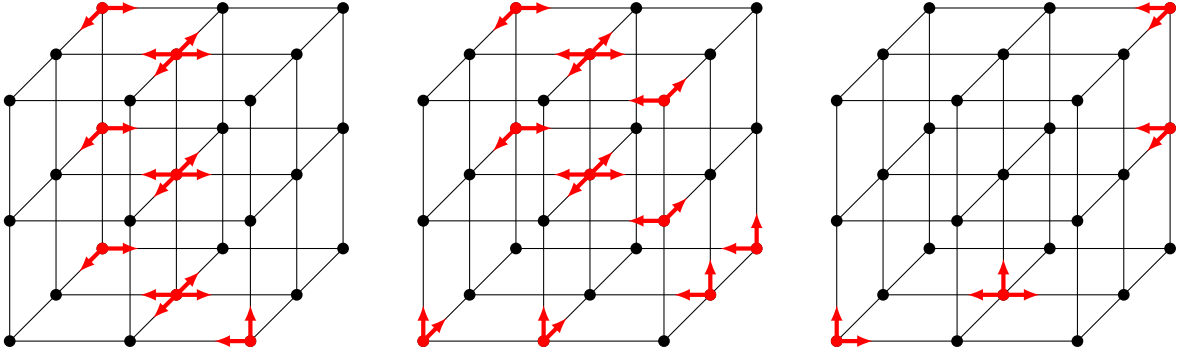
\begin{figure}[H]
\centering
\scalebox{0.4}{\begin{tikzpicture}
\tikzset{>=latex}

\def \dx{4};
\def \dy{4};
\def \dz{4};
\def \nbx{3};
\def \nby{3};
\def \nbz{3};

\foreach \x in {1,...,\nbx} {
    \foreach \y in {1,...,\nby} {
        \foreach \z in {1,...,\nbz} {
            \node at (\x*\dx,\y*\dy,\z*\dz) [circle, fill=black] {};
        }
    }
}

\foreach \x in {1,...,\nbx} {
    \foreach \z in {1,...,\nbz}{
        \foreach \y in {2,...,\nby}{
            \draw [-, color = black, line width = 1pt](\x*\dx,\y*\dy - \dy,\z*\dz) -- ( \x*\dx , \y*\dy, \z*\dz);
        }
    }
}

\foreach \y in {1,...,\nbx} {
    \foreach \z in {1,...,\nbz}{
        \foreach \x in {2,...,\nbx}{
            \draw[-, color = black, line width = 1pt](\x * \dx - \dx,\y*\dy,\z*\dz) -- ( \x * \dx,\y*\dy,\z*\dz);
        }
    }
}

\foreach \x in {1,...,\nbx} {
    \foreach \y in {1,...,\nbz}{
        \foreach \z in {2,...,\nby}{
            \draw[-, color = black, line width = 1pt](\x*\dx,\y*\dy,\z*\dz - \dz) -- ( \x*\dx,\y*\dy,\z*\dz);
        }
    }
}
\node at (\dx, \dx, \dx)[circle, fill=red, color=red]{};
\node at (\dx, 2*\dx, \dx)[circle, fill=red, color=red]{};
\node at (\dx, 3*\dx, \dx)[circle, fill=red, color=red]{};
\node at (2*\dx, 3*\dx, 2*\dx)[circle, fill=red, color=red]{};
\node at (2*\dx, \dx, 2*\dx)[circle, fill=red, color=red]{};
\node at (2*\dx, 2*\dx, 2*\dx)[circle, fill=red, color=red]{};
\node at (3*\dx, \dx, 3*\dx)[circle, fill=red, color=red]{};
\draw [->, color = red, line width = \dx] (\dx,\dx,\dx)--(1.3*\dx, \dx, \dx);
\draw [->, color = red, line width = \dx] (\dx,\dx,\dx)--(\dx, \dx, 1.5*\dx);
\draw [->, color = red, line width = \dx] (\dx,2*\dx,\dx)--(\dx, 2*\dx, 1.5*\dx);
\draw [->, color = red, line width = \dx] (\dx,2*\dx,\dx)--(1.3*\dx, 2*\dx, \dx);
\draw [->, color = red, line width = \dx] (\dx,3*\dx,\dx)--(\dx, 3*\dx, 1.5*\dx);
\draw [->, color = red, line width = \dx] (\dx,3*\dx,\dx)--(1.3*\dx, 3*\dx, \dx);
\draw [->, color = red, line width = \dx] (2*\dx,\dx,2*\dx)--(2.3*\dx, \dx, 2*\dx);
\draw [->, color = red, line width = \dx] (2*\dx,\dx,2*\dx)--(2*\dx, \dx, 2.5*\dx);
\draw [->, color = red, line width = \dx] (2*\dx,2*\dx,2*\dx)--(2*\dx, 2*\dx, 2.5*\dx);
\draw [->, color = red, line width = \dx] (2*\dx,2*\dx,2*\dx)--(2.3*\dx, 2*\dx, 2*\dx);
\draw [->, color = red, line width = \dx] (2*\dx,3*\dx,2*\dx)--(2*\dx, 3*\dx, 2.5*\dx);
\draw [->, color = red, line width = \dx] (2*\dx,3*\dx,2*\dx)--(2.3*\dx, 3*\dx, 2*\dx);
\draw [->, color = red, line width = \dx] (2*\dx,\dx,2*\dx)--(1.7*\dx, \dx, 2*\dx);
\draw [->, color = red, line width = \dx] (2*\dx,\dx,2*\dx)--(2*\dx, \dx, 1.5*\dx);
\draw [->, color = red, line width = \dx] (2*\dx,2*\dx,2*\dx)--(2*\dx, 2*\dx, 1.5*\dx);
\draw [->, color = red, line width = \dx] (2*\dx,2*\dx,2*\dx)--(1.7*\dx, 2*\dx, 2*\dx);
\draw [->, color = red, line width = \dx] (2*\dx,3*\dx,2*\dx)--(2*\dx, 3*\dx, 1.5*\dx);
\draw [->, color = red, line width = \dx] (2*\dx,3*\dx,2*\dx)--(1.7*\dx, 3*\dx, 2*\dx);
\draw [->, color = red, line width = \dx] (3*\dx,\dx,3*\dx)--(3*\dx, 1.3*\dx, 3*\dx);
\draw [->, color = red, line width = \dx] (3*\dx,\dx,3*\dx)--(2.7*\dx, \dx, 3*\dx);
\end{tikzpicture}
\hspace{20mm}
\begin{tikzpicture}
\tikzset{>=latex}

\def \dx{4};
\def \dy{4};
\def \dz{4};
\def \nbx{3};
\def \nby{3};
\def \nbz{3};

\foreach \x in {1,...,\nbx} {
    \foreach \y in {1,...,\nby} {
        \foreach \z in {1,...,\nbz} {
            \node at (\x*\dx,\y*\dy,\z*\dz) [circle, fill=black] {};
        }
    }
}

\foreach \x in {1,...,\nbx} {
    \foreach \z in {1,...,\nbz}{
        \foreach \y in {2,...,\nby}{
            \draw [-, color = black, line width = 1pt](\x*\dx,\y*\dy - \dy,\z*\dz) -- ( \x*\dx , \y*\dy, \z*\dz);
        }
    }
}

\foreach \y in {1,...,\nbx} {
    \foreach \z in {1,...,\nbz}{
        \foreach \x in {2,...,\nbx}{
            \draw[-, color = black, line width = 1pt](\x * \dx - \dx,\y*\dy,\z*\dz) -- ( \x * \dx,\y*\dy,\z*\dz);
        }
    }
}

\foreach \x in {1,...,\nbx} {
    \foreach \y in {1,...,\nbz}{
        \foreach \z in {2,...,\nby}{
            \draw[-, color = black, line width = 1pt](\x*\dx,\y*\dy,\z*\dz - \dz) -- ( \x*\dx,\y*\dy,\z*\dz);
        }
    }
}
\node at (3*\dx, 2*\dx, 3*\dx)[circle, fill=red, color=red]{};
\node at (\dx, 2*\dx, \dx)[circle, fill=red, color=red]{};
\node at (\dx, 3*\dx, \dx)[circle, fill=red, color=red]{};
\node at (2*\dx, 3*\dx, 2*\dx)[circle, fill=red, color=red]{};
\node at (3*\dx, \dx, 2*\dx)[circle, fill=red, color=red]{};
\node at (3*\dx, \dx, \dx)[circle, fill=red, color=red]{};
\node at (\dx, \dx, 3*\dx)[circle, fill=red, color=red]{};
\node at (2*\dx, \dx, 3*\dx)[circle, fill=red, color=red]{};
\node at (2*\dx, 2*\dx, 2*\dx)[circle, fill=red, color=red]{};
\node at (3*\dx, 3*\dx, 3*\dx)[circle, fill=red, color=red]{};
\draw [->, color = red, line width = \dx] (\dx,\dx,3*\dx)--(\dx, \dx, 2.5*\dx);
\draw [->, color = red, line width = \dx] (\dx,\dx,3*\dx)--(\dx, 1.3*\dx, 3*\dx);
\draw [->, color = red, line width = \dx] (2*\dx,\dx,3*\dx)--(2*\dx, \dx, 2.5*\dx);
\draw [->, color = red, line width = \dx] (2*\dx,\dx,3*\dx)--(2*\dx, 1.3*\dx, 3*\dx);
\draw [->, color = red, line width = \dx] (3*\dx,\dx,\dx)--(2.7*\dx, \dx, \dx);
\draw [->, color = red, line width = \dx] (3*\dx,\dx,\dx)--(3*\dx, 1.3*\dx, \dx);
\draw [->, color = red, line width = \dx] (3*\dx,\dx,2*\dx)--(2.7*\dx, \dx, 2*\dx);
\draw [->, color = red, line width = \dx] (3*\dx,\dx,2*\dx)--(3*\dx, 1.3*\dx, 2*\dx);
\draw [->, color = red, line width = \dx] (\dx,2*\dx,\dx)--(\dx, 2*\dx, 1.5*\dx);
\draw [->, color = red, line width = \dx] (\dx,2*\dx,\dx)--(1.3*\dx, 2*\dx,\dx);
\draw [->, color = red, line width = \dx] (2*\dx,2*\dx,2*\dx)--(2*\dx, 2*\dx, 2.5*\dx);
\draw [->, color = red, line width = \dx] (2*\dx,2*\dx,2*\dx)--(2.3*\dx, 2*\dx,2*\dx);
\draw [->, color = red, line width = \dx] (2*\dx,2*\dx,2*\dx)--(2*\dx, 2*\dx, 1.5*\dx);
\draw [->, color = red, line width = \dx] (2*\dx,2*\dx,2*\dx)--(1.7*\dx, 2*\dx,2*\dx);
\draw [->, color = red, line width = \dx] (3*\dx,2*\dx,3*\dx)--(3*\dx, 2*\dx, 2.5*\dx);
\draw [->, color = red, line width = \dx] (3*\dx,2*\dx,3*\dx)--(2.7*\dx, 2*\dx,3*\dx);
\draw [->, color = red, line width = \dx] (\dx,3*\dx,\dx)--(\dx, 3*\dx, 1.5*\dx);
\draw [->, color = red, line width = \dx] (\dx,3*\dx,\dx)--(1.3*\dx, 3*\dx,\dx);
\draw [->, color = red, line width = \dx] (2*\dx,3*\dx,2*\dx)--(2*\dx, 3*\dx, 2.5*\dx);
\draw [->, color = red, line width = \dx] (2*\dx,3*\dx,2*\dx)--(2.3*\dx, 3*\dx,2*\dx);
\draw [->, color = red, line width = \dx] (2*\dx,3*\dx,2*\dx)--(2*\dx, 3*\dx, 1.5*\dx);
\draw [->, color = red, line width = \dx] (2*\dx,3*\dx,2*\dx)--(1.7*\dx, 3*\dx,2*\dx);
\draw [->, color = red, line width = \dx] (3*\dx,3*\dx,3*\dx)--(3*\dx, 3*\dx, 2.5*\dx);
\draw [->, color = red, line width = \dx] (3*\dx,3*\dx,3*\dx)--(2.7*\dx, 3*\dx,3*\dx);

\end{tikzpicture}
\hspace{20mm}
\begin{tikzpicture}
\tikzset{>=latex}

\def \dx{4};
\def \dy{4};
\def \dz{4};
\def \nbx{3};
\def \nby{3};
\def \nbz{3};

\foreach \x in {1,...,\nbx} {
    \foreach \y in {1,...,\nby} {
        \foreach \z in {1,...,\nbz} {
            \node at (\x*\dx,\y*\dy,\z*\dz) [circle, fill=black] {};
        }
    }
}

\foreach \x in {1,...,\nbx} {
    \foreach \z in {1,...,\nbz}{
        \foreach \y in {2,...,\nby}{
            \draw [-, color = black, line width = 1pt](\x*\dx,\y*\dy - \dy,\z*\dz) -- ( \x*\dx , \y*\dy, \z*\dz);
        }
    }
}

\foreach \y in {1,...,\nbx} {
    \foreach \z in {1,...,\nbz}{
        \foreach \x in {2,...,\nbx}{
            \draw[-, color = black, line width = 1pt](\x * \dx - \dx,\y*\dy,\z*\dz) -- ( \x * \dx,\y*\dy,\z*\dz);
        }
    }
}

\foreach \x in {1,...,\nbx} {
    \foreach \y in {1,...,\nbz}{
        \foreach \z in {2,...,\nby}{
            \draw[-, color = black, line width = 1pt](\x*\dx,\y*\dy,\z*\dz - \dz) -- ( \x*\dx,\y*\dy,\z*\dz);
        }
    }
}
\node at (\dx, \dx, 3*\dx)[circle, fill=red, color=red]{};
\node at (2*\dx, \dx, 2*\dx)[circle, fill=red, color=red]{};
\node at (3*\dx, 2*\dx, \dx)[circle, fill=red, color=red]{};
\node at (3*\dx, 3*\dx, \dx)[circle, fill=red, color=red]{};
\draw [->, color = red, line width = \dx] (\dx,\dx,3*\dx)--(\dx, 1.3*\dx, 3*\dx);
\draw [->, color = red, line width = \dx] (\dx,\dx,3*\dx)--(1.3*\dx, \dx, 3*\dx);
\draw [->, color = red, line width = \dx] (2*\dx,\dx,2*\dx)--(2*\dx, 1.3*\dx, 2*\dx);
\draw [->, color = red, line width = \dx] (2*\dx,\dx,2*\dx)--(1.7*\dx, \dx, 2*\dx);
\draw [->, color = red, line width = \dx] (2*\dx,\dx,2*\dx)--(2.3*\dx, \dx, 2*\dx);
\draw [->, color = red, line width = \dx] (3*\dx,2*\dx,\dx)--(3*\dx, 2*\dx, 1.5*\dx);
\draw [->, color = red, line width = \dx] (3*\dx,2*\dx,\dx)--(2.7*\dx, 2*\dx, \dx);
\draw [->, color = red, line width = \dx] (3*\dx,3*\dx,\dx)--(3*\dx, 3*\dx, 1.5*\dx);
\draw [->, color = red, line width = \dx] (3*\dx,3*\dx,\dx)--(2.7*\dx, 3*\dx, \dx);
\end{tikzpicture}}
\caption{ $a_{3, 3, 2}=7, b_{3, 3, 2}=10, c_{3, 3, 2}=4$}
\end{figure}

Previous research studies the case when $k=\ell$. In particular, $a_{n,k,k}$ corresponds to a $1$-covering while $b_{n,k,k}$ and $c_{n,k,k}$ correspond to a $1$-packing and $2$-packing, respectively. Furthermore $c_{q, k, k}$ corresponds to generalized $q$-ary Hamming-distance-$3$ subsets of $H_{q, k}$, which are useful for error correcting codes. The most classical bound in the case of coverings is the sphere-packing bound. We give the analog in this case; our proof is nearly identical to the classical one. This determines $a_{n,k,l}$ to within a constant depending on $l$.
\begin{thm}\label{SpherePacking}
 We have
\[
\frac{n^k}{\ell(n-1)+1}\le a_{n,k,l}\le n^{k-1}.
\]
\end{thm}
\begin{proof}
Suppose for the sake of contradiction, there exists a covering $S$ of $H_{n,k}$ with $\ell$-rooks and $|S|<\frac{n^k}{\ell(n-1)+1}$. Since each rook covers at most $\ell(n-1)+1$ points, it follows that $S$ covers at most $ |S|(\ell(n-1)+1)<n^k$ points, which is a contradiction. To prove the upper bound, let S be the set of all points with first coordinate 0. Allow each point in S to attack in the direction of the first coordinate, and arbitrarily choose the other $\ell-1$ directions in which it may attack and these rooks collectively cover the cube. Note that we do not utilize the last $\ell-1$ dimensions for the upper bound.
\end{proof}
Note that since the above theorem holds for $\ell=1$ and it implies that $a_{n,k,1}=n^{k-1}$. Given the triviality of this case, we consider $\ell\ge 2$ in the remainder of the paper. The analogous lower bounds for $b_{n,k,k}$ and $c_{n,k,k}$ comes from the classical Singleton bound \cite{singleton1964maximum}. The proof presented in the classical case can be adapted to this situation as well, however we rely on a more geometrical argument.
\begin{thm}\label{Singleton}
For all positive integers $n$, $k$, and $\ell$ with $k\ge \ell$, we have
\[
b_{n,k,\ell}\le \frac{k n^{k-1}}{\ell}.
\]
Furthermore, if $k\ge\ell\ge 2$ then 
\[
c_{n,k,\ell}\le \frac{\binom{k}{2} n^{k-2}}{\binom{\ell}{2}}.
\]
\end{thm}
\begin{proof}
For $b_{n,k,\ell}$, consider all lines parallel to edges of the $H_{n,k}$ containing $n$ points in $H_{n,k}$. Note that there are $kn^{k-1}$ such lines by choosing a direction and letting the remaining coordinates vary over all possibilities within the cube. Furthermore, no two $\ell$-rooks can cover the same axis. Since each $\ell$-rook cover $\ell$ axes, it follows that $b_{n,k,\ell}\le \frac{k n^{k-1}}{\ell}.$ Similarly, for $c_{n,k,\ell}$ consider all planes passing through $H_{n,k}$, parallel to one of the faces. Note there are $\binom{k}{2}n^{k-2}$ of these faces and each $\ell$-rook covers $\binom{\ell}{2}$ planes. If two rooks cover the same plane, then they intersect, and it follows that  
$
c_{n,k,\ell}\le \frac{\binom{k}{2} n^{k-2}}{\binom{\ell}{2}}
$
for $\ell\ge 2$. (If $\ell=1$, the $\ell$-rook does not determine a plane, so the proof does not follow.)
\end{proof}
Note that $c_{n,k,1}\le n^{k-1}$ as each $1$-rook covers $n$ points and the points these rooks cover are distinct. This can be achieved by putting $1$-rooks on all points with the first coordinate $0$ and having all rooks point in the direction of the first coordinate. Given this difference in behavior between $\ell\ge 2$ and $\ell=1$ for $c_{n,k,\ell}$, we assume that $\ell\ge 2$ for the remainder of the paper in this case as well.

In the remainder of the paper, we focus on the asymptotic growth rates of $a_{n,k,\ell}$, $b_{n,k,\ell}$, and $c_{n,k,\ell}$ when $k$ and $\ell$ are fixed and $n$ increases. 

\begin{notation}
Let $a_{k,\ell}=\displaystyle\lim_{n\to\infty}\frac{a_{n,k,\ell}}{n^{k-1}},$
$b_{k,\ell}=\displaystyle\lim_{n\to\infty}\frac{b_{n,k,\ell}}{n^{k-1}},$
and $c_{k,\ell}=\displaystyle\lim_{n\to\infty}\frac{c_{n,k,\ell}}{n^{k-2}}.$
\end{notation}
The remainder of the paper is organized as follows. Section $2$ establishes the existence of such limits for all $k$ and $\ell$ (with $\ell\ge 2$ for $c_{k,\ell}$). Section 3 focuses on covering bounds and demonstrates that for $\ell\neq 1$ that the lower sphere-packing bound in Theorem 2 is never asymptotically tight. Furthermore, Section 3 proves that $a_{k,\ell}\to\frac{1}{\ell}$ as $k\to\infty$. Section 4 focuses on the packing bounds and demonstrates that $b_{k,\ell}$ and $c_{k,\ell}$ achieve the bounds in Theorem \ref{Singleton} in several possible cases. Finally, Section 5 presents a series of open problems regarding $a_{k,l}, b_{k,l},$ and $c_{k,l}$. 

\section{Existence Results}
The general idea for our proofs in this section is to demonstrate that $a_{nm,k,\ell}\le m^{k-1} a_{n,k,\ell}$ for all integers $m$ and then show that adjacent terms are sufficiently close. (The first inequality is reversed for $b_{n,k,\ell}$ and a similar result holds for $c_{n,k,\ell}$.) For $a_{n,k,\ell}$ and $b_{n,k,\ell}$, the first inequality is demonstrated using a construction of Blokhuis and Lam \cite{blokhuis1984more} whereas for $c_{n,k,\ell}$ we rely on a different construction.
\begin{thm}\label{aexists}
For positive integers $k\ge \ell$, the limits \[a_{k,\ell}=\lim_{n\to\infty}\frac{a_{n,k,\ell}}{n^{k-1}},\]
\[b_{k,\ell}=\lim_{n\to\infty}\frac{b_{n,k,\ell}}{n^{k-1}}\] exist.
\end{thm}
\begin{proof}
We first consider $a_{k,\ell}$. For $\ell=1$, $a_{n,k,1}=n^{k-1}$ and the result is trivial. Therefore it suffices to assume that $k\ge 2$. Using Theorem \ref{SpherePacking}, it follows that 
\[\frac{1}{\ell}\le\liminf_{n\to\infty}\frac{a_{n,k,\ell}}{n^{k-1}}\le\limsup_{n\to\infty}\frac{a_{n,k,\ell}}{n^{k-1}}\le 1.\] Now suppose that $L=\liminf_{n\to\infty}\frac{a_{n,k,\ell}}{n^{k-1}}$. Then for every $\epsilon>0$, there exists an integer $m$ such that $\frac{a_{m,k,\ell}}{m^{k-1}}\le L+\frac{\epsilon}{2}$. Now consider the points $(x_1,\ldots,x_k)$ in $\{0,1,\ldots,n-1\}^k$ such that \[x_1+\cdots+x_k\equiv 0\mod n.\] (This is the construction present in \cite{blokhuis1984more}.) Note that if a $k$-rook is placed at every point in this construction, all points are covered and that every point of an outer face of the hypercube has a axis ``protruding" out of it. Therefore we can essentially blowup every point in $H_{m,k}$ to a copy of $H_{n,k}$ to create an $H_{mn,k}$, mark all the corresponding $H_{n,k}$ in $H_{mn,k}$ that correspond to rooks from the construction of $a_{m,k,\ell}$ and place $\ell$-rooks within these $H_{n,k}$ corresponding to the points from the earlier construction. Then choose the $\ell$ axes for each of these rooks that corresponds to the orientation for the $\ell$-rook in the original construction of $H_{m,k}$ in $a_{m,k,\ell}$. This gives a covering of $H_{mn,k}$, so it follows that $a_{nm,k,\ell}\le n^{k-1} a_{m,k,\ell}$.

Now consider $a_{n+1,k,\ell}$ and $a_{n,k,\ell}$. If we let $H_{n,k}=\{0,\ldots,n-1\}^{k}$ and $H_{n+1,k}=\{0,\ldots,n\}^{k}$ then we place the construction for $a_{n,k,\ell}$ in $\{1,\ldots,n\}^{k}$. In order to cover the rest of the cube, place $\ell$-rooks at every point with at least $2$-coordinates being $0$ and we choose the directions of the points arbitrarily. For the remaining $k(n-1)^{k-1}$ points with exactly one $0$ and we break into cases with points of the form $(a_1,\ldots,a_{i-1},0,a_{i+1},\ldots,a_k)$. In order to cover these point we take all points such points with $a_{i}=0$ and place one axis of the $\ell$ possible in the direction of the $(i+1)^{st}$ coordinate where indices are taken$\mod n$. These points together cover the $H_{n+1,k}$ and we have added on at most $kn^{k-2}+\sum_{i=2}^{k}\binom{k}{i}n^{k-i}\le \sum_{i=1}^{k}\binom{k}{i}n^{k-2}\le 2^{k}n^{k-2}$ additional points. Therefore it follows that $a_{n+1,k,\ell}\le 2^{k}n^{k-2}+a_{n,k,\ell}$ and thus \[\frac{a_{n+1,k,\ell}}{(n+1)^{k-1}}\le \frac{2^{k}n^{k-2}+a_{n,k,\ell}}{(n+1)^{k-1}}\]\[\le \frac{2^{k}n^{k-2}+a_{n,k,\ell}}{n^{k-1}}=\frac{2^k}{n}+\frac{a_{n,k,\ell}}{n^{k-1}}.\]
Taking $n$ sufficiently large it follows that \[\sum_{i=mn}^{mn+m-1}\frac{2^k}{i}<\frac{\epsilon}{2}.\] Thus for $i\ge mn$ it follows that $\frac{a_{i,k,\ell}}{i^{k-1}}\le L+\epsilon$. Therefore \[\limsup_{n\to\infty}\frac{a_{n,k,\ell}}{n^{k-1}}\le L+\epsilon\] and since $\epsilon$ was an arbitrary constant greater than $0$, the result follows. 
For $b_{k,\ell}$, an identical procedure demonstrates that $\frac{b_{mn,k,\ell}}{(mn)^{k-1}}\ge \frac{b_{n,k,\ell}}{(n)^{k-1}}$ for all positive integers $m$ and $n$. Furthermore the sequence $\frac{b_{n,k,\ell}}{n^{k-1}}$ is bounded due to Theorem \ref{Singleton} and note that \[\frac{b_{n+1,k,\ell}}{(n+1)^{k-1}}\ge \frac{b_{n,k,\ell}}{(n+1)^{k-1}}= \frac{n^{k-1}}{(n+1)^{k-1}}\bigg(\frac{b_{n,k,\ell}}{n^{k-1}}\bigg).\] Thus taking $L=\limsup_{n\to\infty}\frac{b_{n,k,\ell}}{n^{k-1}}$ and choosing $\epsilon>0$ arbitrarily there exists an $m$ such that $\frac{b_{m,k,\ell}}{m^{k-1}}>L-\frac{\epsilon}{2}$. Now suppose that $m$ satisfies $(\frac{mn}{mn+n-1})^k>\frac{L-\frac{\epsilon}{2}}{L-\epsilon}$. Then for all $i\ge mn$, $\frac{a_{i,k,\ell}}{i^{k-1}}>L-\epsilon$. Therefore, 
\[\limsup_{n\to\infty}\frac{b_{n,k,\ell}}{n^{k-1}}>L-\epsilon.\] Since $\epsilon$ was arbitrary the result follows.
\end{proof}
For the existence of $c_{k,\ell}$, we follow a similar strategy except we rely on a different construction for the initial inequality that allows only for prime ``blowup" factors. This construction is closely related and motivated by the construction of general $q$-ary codes. 

\begin{thm}\label{cExists}
For positive integers $k\ge \ell\ge 2$, the limit \[c_{k,\ell}=\lim_{n\to\infty}\frac{c_{n,k,\ell}}{n^{k-2}}\] exists.
\end{thm}
\begin{proof}
Suppose $p$ is prime and $p\ge k$. Consider the set of points $(x_1,\ldots,x_k)$ in $H_{p,k}$ that satisfy  $x_{k-1}\equiv x_1+\cdots+x_{k-2} \mod p$ and $x_k\equiv x_1+2x_2+3x_3+\cdots+(k-2)x_{k-2}\mod p$. We will show that in this construction no two points are less than distance $3$ apart. Suppose for sake of contradiction that there are two points $A=(a_1,\ldots,a_k)$ and $B=(b_1,\ldots,b_k)$ such that the distance between A and B is at most 2. If $a_i=b_i$ for all $t$ with $1\leq t\leq k-2$, then $A=B$. If $a_t=b_t$ for $1\le t\le k-2$ except for $i\in \{1,\ldots, k-2\}$ where $a_i\neq b_i$, then $a_{k-1}\neq b_{k-1}$ and $a_{k}\neq b_{k}$. Finally, we consider the case where $a_t=b_t$ for $1\le t\le k-2$ except for $i,j\in \{1,\ldots, k-2\}$ where $a_i\neq b_i$ and $a_j\neq b_j$. If both of the last two digits match then $a_i+a_j\equiv b_i+b_j\mod p$ and $ia_i+ja_j\equiv ib_i+jb_j\mod p$. Subtracting $i$ times the first equation from the second yields $(j-i)a_j\equiv (j-i)b_j\mod p$ or $a_j\equiv b_j\mod p$, which is impossible. Thus each pair of points in $S$ differ by at least a distance $3$. Furthermore note the set $S$ has exactly $p^{k-2}$ points.
 
Now given a construction for $c_{n,k,l}$ in $H_{n,k}$, we can blow up each point to a copy of $H_{p,k}$ (for $p>k$ and $p$ prime). Then place the construction given above into each $H_{p,k}$ corresponding to marked points in the original construction. Orienting the set of points in each $H_{p,k}$ to match the original orientation of the corresponding point in $H_{n,k}$, it follows that $\frac{H_{np,k}}{(np)^{k-2}}\ge \frac{H_{n,k}}{n^{k-2}}$ for all primes greater than $k$. 
Furthermore, note that \[\frac{c_{n+1,k,\ell}}{(n+1)^{k-2}}\ge \frac{c_{n,k,\ell}}{(n+1)^{k-2}}= \frac{n^{k-2}}{(n+1)^{k-2}}\bigg(\frac{c_{n,k,\ell}}{n^{k-2}}\bigg).\] 
Now $\frac{c_{n,k,l}}{n^{k-2}}$ is bounded above due to Theorem \ref{Singleton} and bounded below as it is positive. Let $L=\limsup_{n\to\infty}\frac{c_{n,k,\ell}}{n^{k-2}}$ and thus for every $\epsilon>0$ there is an $m$ such that $\frac{c_{n,k,\ell}}{n^{k-2}}>L-\frac{\epsilon}{2}$. Now order the primes $2=p_1<p_2<\cdots$. Since $\lim_{i\to\infty}\frac{p_{i+1}}{p_{i}}=1$ it follows that there exists $j$ such that for $i\ge j$, $\frac{p_{i+1}}{p_i}<(\frac{L-\frac{\epsilon}{2}}{L+\epsilon})^{\frac{1}{k-2}}$. For every integer $t>p_jn$ it follows that $t\in [p_in,p_{i+1}n-1]$ and  $\frac{c_{t,k,\ell}}{t^{k-2}}>(\frac{t}{p_in})^{k-2}\frac{c_{p_in,k,\ell}}{(p_in)^{k-2}}>L-\epsilon$. Therefore $\lim\inf_{n\to\infty}\frac{c_{n,k,\ell}}{n^{k-2}}>L-\epsilon$, and since $\epsilon$ was arbitrary the result follows.
\end{proof}

\section{Bounds for Covering}
Given the initial bounds from Theorem \ref{SpherePacking}, it follows that $\frac{1}{\ell}\le a_{k,\ell}\le 1$. However, in general we demonstrate that $a_{k,\ell}\neq \frac{1}{\ell}$, except for the trivial case $a_{k,1}=1$. To do this it is necessary to ``amortize" a result of Rodemich \cite{rodemich1970coverings} which is equivalent to $a_{n,k,k}\ge \frac{n^{k-1}}{k-1}$. However the original proof given by Rodemich can be replicated for this situation and we reproduce the proof below for the readers convenience.

\begin{thm}\label{KeyIdea}
Suppose that $N\le n^{k-1}$. Then for sufficiently large $n$, $N$ $k$-rooks on a $H_{n, k}$ cover at most $kNn-\frac{(k-1)N^2}{n^{k-2}}$ points.
\end{thm}

\begin{proof}
The bound is clear when $k=1$. For $k=2$ note that $N$ $2$-rooks cover at most $n^2-(n-N)^2=2Nn-N^2$ points as least $n-N$ rows and columns are uncovered. Therefore it suffices to consider $k\ge 3$. Furthermore, when $N\in [\frac{n^{k-1}}{k-1}, n^{k-1}]$, we have $kNn-\frac{(k-1)N^2}{n^{k-2}}\ge n^{k}$ so the bound holds in these cases. Hence, it suffices to consider $N\le \frac{n^{k-1}}{k-1}$.

Now consider any set $S$ of $k$-rooks with $|S|=N$. For any point $P\in H_{n,k}$ define $c_{j}(P)$ to be the number of points of times that the point $P$ is attacked in the $j^{th}$ direction. Furthermore define $q(P)$ the number of directions that $P$ is attacked in and define
\[m(P)=\sum_{1\le j\le k}c_j(P)=q(P)+\sum_{c_j(P)>0}(c_j(P)-1).\] Furthermore define $e_{i,j}(P)$ to be $1$ if $P$ is covered in the $i$ and $j$ directions and $0$ otherwise. Then note that 
\[\sum_{1\le i<j\le k}e_{i,j}(P)=\frac{q(P)(q(P)-1)}{2}\le \frac{k(q(P)-1)}{2}\]
for points $P$ that are attacked and therefore
\[q(P)\ge 1+\frac{2}{k}\sum_{1\le i<j\le k}{e_{i,j}(P)}.\] Finally define $n_j(P)=c_j(P)-1$ if $c_j(P)$ is positive and $0$ otherwise. Therefore
\begin{align*}m(P)&=q(P)+\sum_{1\le j\le k}n_j(P)
\\ &\ge 1+\sum_{1\le j\le k}n_j(P)+\frac{2}{k}\sum_{1\le i<j\le k}{e_{i,j}(P)}
\end{align*}
for points $P$ that are attacked and suppose that $S$ attacks the points $T\in H_{n,k}$. Summing over $P\in T$ yields
\begin{align*}kNn &\ge |T|+\sum_{1\le j\le k}\sum_{P\in T}n_j(P)+\frac{2}{k}\sum_{1\le i<j\le k}\sum_{P\in T}e_{i,j}(P)
\\ &=|T|+\sum_{1\le j\le k}n_j+\frac{2}{k}\sum_{1\le i<j\le k}e_{i,j}
\end{align*}
where we have defined
\[n_j=\sum_{P\in T}n_j(P)\] and 
\[e_{i,j}=\sum_{P\in T}e_{i,j}(P).\] Now we arbitrarily order the $n^{k-2}$ planes in the $(i,j)$ direction. For $r^{th}$ plane suppose there are $a_r$ rows in the $i^{th}$ direction with a point of $S$ in them, $b_r$ rows in the $j^{th}$ direction with a point of $S$ in them, and $d_r$ total points in this plane. Furthermore for convenience define $\alpha_r=d_r-a_r$ and $\beta_r=d_r-b_r$. Then it follows that 
\begin{align*}e_{i,j}&=\sum_{1\le r\le n^{k-2}}{a_rb_r}\\ &=\sum_{1\le r\le n^{k-2}}(d_r-\alpha_r)(d_r-\beta_r)
\\ &=\sum_{1\le r\le n^{k-2}}\bigg(d_r-\frac{\alpha_r+\beta_r}{2}\bigg)^2-\bigg(\frac{\alpha_r-\beta_r}{2}\bigg)^2\end{align*}
Using the trivial inequality that $|\alpha_r-\beta_r|\le n$ it follows that
\begin{align*}e_{i,j} &\ge\frac{1}{n^{k-2}}\bigg(\sum_{1\le r\le n^{k-2}}d_r-\frac{\alpha_r+\beta_r}{2}\bigg)^2-\frac{n}{2}\sum_{1\le r\le n^{k-2}}\frac{\alpha_r+\beta_r}{2}
\\ &=\frac{1}{n^{k-2}}\bigg(N-\frac{n_i+n_j}{2n}\bigg)^2-\frac{n_i+n_j}{4}.\end{align*}
Here we have used the fact that
\[n\sum_{1\le r\le n^{k-2}}\alpha_r+\beta_r=n_i+n_j\]
which follows from counting the number of points covered multiple times in the $i$th and $j$th directions.
Summing over all $i,j$ it follows that 
\begin{align*}\sum_{1\le i<j\le k}e_{i,j}&\ge \frac{k(k-1)N^2}{2n^{k-2}}-\frac{(k-1)N}{n^{k-1}}\sum_{1\le j\le k}n_j-\frac{k-1}{4}\sum_{1\le j\le k}n_j+\frac{1}{4n^k}\sum_{1\le i<j\le k}(n_i+n_j)^2.\end{align*}
Applying this inequality it follows that 
\begin{align*}kNn &\ge |T|+\sum_{1\le j\le k}n_j+\frac{2}{k}\sum_{1\le i<j\le k}e_{i,j}
\\ &\ge |T|+\bigg(1-\frac{2(k-1)N}{kn^{k-1}}-\frac{k-1}{2k}\bigg)\sum_{1\le j\le k} n_j+\frac{(k-1)N^2}{n^{k-2}}+\frac{1}{2kn^k}\sum_{1\le i<j\le k}(n_i+n_j)^2
 \\ &\ge |T|+\bigg(1-\frac{2(k-1)N}{kn^{k-1}}-\frac{k-1}{2k}\bigg)\sum_{1\le j\le k} n_j+\frac{(k-1)N^2}{n^{k-2}}.\end{align*}
Using $N\le \frac{n^{k-1}}{k-1}$ it then follows that
\begin{align*}kNn &\ge |T|+\bigg(1-\frac{2}{k}-\frac{k-1}{2k}\bigg)\sum_{1\le j\le k} n_j+\frac{(k-1)N^2}{n^{k-2}}
\\ &\ge |T|+\bigg(\frac{k-3}{2k}\bigg)\sum_{1\le j\le k} n_j+\frac{(k-1)N^2}{n^{k-2}}
\\ &\ge |T|+\frac{(k-1)N^2}{n^{k-2}}\end{align*}
and therefore it follows that 
\[|T|\le kNn-\frac{(k-1)N^2}{n^{k-2}}\] as desired.
\end{proof}
Note the previous bound in general cannot be improved as $a_{k+1,k+1}=\frac{1}{k}$ when $k$ is a prime power due to the existence of perfect codes \cite{blokhuis1984more}. Using this amortized version of Rodemich's result, we now prove a better lower bound for $a_{k,\ell}$. 
\begin{thm}
For every pair of positive integers $(\ell, k)$ with $\ell\le k$, we have
\[a_{k, \ell}\ge \frac{2}{\ell(1+\sqrt{1-\frac{4(\ell-1)}{\ell^2\binom{k}{\ell}}})}.\]
\end{thm}
\begin{proof}
Suppose we have a configuration of $N$ $\ell$-rooks that covers $H_{n, k}$. Since the $\ell=k$ case of Theorem $8$ is established in by Rodemich's result \cite{rodemich1970coverings}, it suffices to consider when $k>\ell$. In this case $\binom{k}{\ell}>1$, so the right-hand side above is less than $\frac{1}{\ell-1}$. Therefore, it suffices to consider the case $N\le \frac{n^{k-1}}{\ell-1}$. We first prove the following lemma:
\begin{lemma}
Suppose that $a_1, \ldots, a_{n^{k-\ell}}$ are nonnegative reals that satisfy $\displaystyle\sum_{i=1}^{n^{k-\ell}}a_i=A\le \frac{n^{k-1}}{\ell-1}$. Then 
\[\sum_{i, a_i\le \frac{n^{\ell-1}}{\ell-1}}(\ell n a_i-\frac{\ell-1}{n^{\ell-2}}a_i^2)+\sum_{i, a_i>\frac{n^{\ell-1}}{\ell-1}}n^\ell\le \ell nA-\frac{\ell-1}{n^{k-2}}A^2.\]
\end{lemma}
\begin{proof}
Consider the piecewise function $f(x)$ defined by
\[f(x)=\begin{cases}\ell n x-\frac{\ell-1}{n^{\ell-2}}x^2 & x\le \frac{n^{\ell-1}}{\ell-1};
\\ n^\ell & x\ge \frac{n^{\ell-1}}{\ell-1}.
\end{cases}\]
Then $f(x)$ is continuous and concave on the region $[0, A]$. It follows that for $A=\sum_{i=1}^{n^{k-\ell}}a_i$ fixed, the left-hand side achieves its minimum when the $a_i$ are all equal to $\frac{A}{n^{k-\ell}}$. Since $\frac{A}{n^{k-\ell}}\le \frac{n^{\ell-1}}{\ell-1}$, it follows that the left-hand side is at most
\[n^{k-\ell}f(\frac{A}{n^{k-\ell}})=\ell n A-\frac{\ell-1}{n^{k-2}}A^2\]
as required. 
\end{proof}
Now we proceed with the proof of Theorem 10. We consider the $\binom{k}{\ell}$ possible choices of direction for the $\ell$-points separately. Label these, directions $1, \ldots, \binom{k}{\ell}$ arbitrarily. Each choice of direction corresponds to a choice of $\ell$ out of $k$ coordinates, so there are $n^{k-\ell}$ distinct dimension-$\ell$ hypercubes for each direction, and these collectively form a partition of the full $H_{n, k}$. Order these $\ell$-dimensional hyperplanes arbitrarily and let $a_{i, j}$ denote the number of $\ell$-points in the $j^{th}$ hyperplane of the $i^{th}$ direction with points in that direction. Furthermore let $A_i=\sum_{j=1}^{n^{k-\ell}}a_{i, j}$. Since the $\binom{k}{\ell}$ directions contain all points exactly once between them, $\sum_{i=1}^{\binom{k}{l}}a_i=N$. Also, since $N\le \frac{n^{k-1}}{\ell-1}$, we have $A_i\le \frac{n^{k-1}}{\ell-1}$ for each $i$. Now invoking Theorem $6$, the total number of cubes covered is bounded above by
\[\sum_{i=1}^{\binom{k}{\ell}}(\sum_{j, a_{i, j}\le \frac{n^{\ell-1}}{\ell-1}}(\ell n a_{i, j}-\frac{\ell-1}{n^{\ell-2}}a_{i, j}^2)+\sum_{j, a_{i, j}>\frac{n^{\ell-1}}{\ell-1}}n^\ell).\]
It follows that
\begin{align*}
n^k&\le \sum_{i=1}^{\binom{k}{\ell}}(\sum_{j, a_{i, j}\le \frac{n^{\ell-1}}{\ell-1}}(\ell n a_{i, j}-\frac{\ell-1}{n^{\ell-2}}a_{i, j}^2)+\sum_{j, a_{i, j}>\frac{n^{\ell-1}}{\ell-1}}n^\ell)
\\ &\le \sum_{i=1}^{\binom{k}{\ell}}(\ell nA_i-\frac{\ell-1}{n^{k-2}}A_i^2)
\\ &= \ell nN-\frac{\ell-1}{n^{k-2}}\sum_{i=1}^{\binom{k}{\ell}}A_i^2
\\ &\le \ell nN-\frac{\ell-1}{\binom{k}{\ell}n^{k-2}}N^2
\end{align*}
where we have used Lemma 9 and then the Cauchy-Schwarz inequality. Rearranging this gives
\[(\ell-1)(\frac{N}{n^{k-1}})^2-\binom{k}{l}(\frac{\ell N}{n^{k-1}}-1)\le 0.\]
It follows that for all $n$,
\[a_{n, k, \ell} \ge \frac{2n^{k-1}}{\ell(1+\sqrt{1-\frac{4(\ell-1)}{\ell^2\binom{k}{\ell}}})},\]
and the result follows. 
\end{proof}
\begin{corollary}
For $k\ge \ell\ge 2$, $a_{k,\ell}\neq \frac{1}{\ell}$. Therefore, in the limit, $a_{n,k,\ell}$ never achieves the lower bound of the sphere-packing bound.
\end{corollary}
However, despite the fact that $a_{k,\ell}\neq \frac{1}{\ell}$ for $\ell\ge 2$, we can show that as $k$ gets large $a_{k,\ell}$ in fact approaches $\frac{1}{\ell}$. In particular the portion of forced ``overlap" of the attacking rooks goes to $0$. For convenience, define $f(k)$ to be the largest prime power less than or equal to $k$. 
\begin{thm}\label{intLower}
For every pair of positive integers $(k, \ell)$, with $k\ge 2$  and $f(k)\ge \ell$,
\[a_{k, \ell}\le \frac{1}{f(k)-1}\left\lceil\frac{f(k)}{\ell}\right\rceil.\]
\end{thm}
\begin{proof}
Take an integer $n_1>\frac{f(k)}{\ell}$. We first construct a size-$n_1$, dimension-$k$ block from $\left\lceil\frac{f(k)}{\ell}\right\rceil$ $\ell$-rooks. In particular consider points that satisfy $x_1+x_2+\cdots+x_k\equiv i\mod n_1$ for $0\le i\le \lceil\frac{f(k)}{\ell}\rceil-1$ and then choosing the $(i\ell+1)^{st}$ through $((i+1)\ell)^{th}$ directions to attack for the points whose coordinate sum is equivalent to $i$ where we take the specified direction $\mod n$. Note that this block has an attacking line in every possible axis.

Since perfect $q$-nary covering codes exists for prime powers $q$ (see \cite{van1991bounds} for example) and $f(k)\le k$, we have $a_{k, k}\le a_{f(k), f(k)}=\frac{1}{f(k)-1}$. Now we note using the size $n_1$ scaled blocks in place of points for a construction of $a_{n_2,k,k}$, an $H_{n_1n_2, k}$ can be tiled with at most
\[(\left\lceil\frac{f(k)}{\ell}\right\rceil n_1^{k-1})(a_{n_2, k, k})\]
$\ell$-rooks, and the result follows from $\lim_{n\to\infty}\frac{a_{n, k, k}}{n^{k-1}}=a_{k, k}$. 
\end{proof}
\begin{corollary}
For every positive integer $\ell$, $\lim_{k\to\infty}a_{k, \ell}=\frac{1}{\ell}$. 
\end{corollary}
We end the section on packing with the specific case of $a_{3,2}$ that demonstrates that the bounds in the previous two theorems are not tight in general. 
\begin{thm} It holds that
\[a_{3, 2}\le \frac{1}{\sqrt{2}}.\]
\end{thm} 
Note that this bound is less than $\frac{3}{4}$, the bound achieved by Theorem \ref{intLower}.
\begin{proof}
Let $(a, b)$ be any pair of positive integers that satisfies $2<\frac{a}{b}<\sqrt{2}+1$, so that $\frac{4ab}{2a-2b}\ge a+b$. Consider a construction on $H_{2a+2b, 3}$. For $0\le i\le 2a-1, 0\le j\le 2b-1$ we place a $2$-rook at $(i, j, \lfloor\frac{2bi+j}{2a-2b}\rfloor)$ that covers along the second and third coordinates and place a $2$-rook at $(2a+2b-j-1, 2a+2b-i-1, 2a+2b-1-\lfloor\frac{2bi+j}{2a-2b}\rfloor)$ which attacks along the first and third coordinates. We note that the points between these groups are distinct since the first coordinates between them never coincide. 

 Now we claim that the uncovered squares in the plane $z=k$ are contained in the union of $2b$ columns and $2b$ rows. Indeed, in each plane of this form, either $2a-2b$ rooks of the first type are covering in the second coordinate, or $2a-2b$ rooks of the second type are covering the third direction. Since the corresponding rooks are in distinct rows, the remaining plane can be covered via at most $2b$ $2$-rooks, so this construction yields a covering of $H_{2a+2b, 3}$ with at most $8ab+2b(2a+2b)=4b^2+12ab$ $2$-rooks. This yields an upper bound $a_{3, 2}\le \frac{4b^2+12ab}{4(a+b)^2}=\frac{1+3t}{(1+t)^2}$ where $t=\frac{a}{b}$ as the proof of Theorem \ref{aexists} implies $\frac{a_{n,3,2}}{n^2}\ge a_{3,2}$. Taking $t=\frac{a}{b}$ to be an arbitrarily precise rational approximation of $\sqrt{2}+1$ from below, we obtain
\[a_{3, 2}\le \frac{4+3\sqrt{2}}{(2+\sqrt{2})^2}=\frac{1}{\sqrt{2}}\]
as required.

\begin{thm}
It holds that
\[\frac{9-3\sqrt{5}}{4}\le a_{3,2}\].
\end{thm}
In order to prove this lower bound, we first introduce some algebraic lemmas:
\begin{lemma}
Suppose that $c_i, x_i$ are nonnegative integers with $c_i\in [0, n], x_i\in [0, n^2]$ for $1\le i\le n$ and set $C=\sum_{i=1}^nc_i$ and $X=\sum_{i=1}^nx_i$. Suppose that  $(n-c_i-x_i)(n-c_i)\le X$ for each $i$. Then:
\[C+X\ge n^2(\frac{t}{2}+1-\frac{t(1+t)}{2(1-t)}-(1-\frac{t}{1-t})\sqrt{t})-2n\]
where $t=\frac{X}{n^2}$. 
\end{lemma}
\begin{proof}
 Suppose we have $x_i, c_i$ for which the minimum value of $X+C$ is achieved. We first claim that it suffices to prove the statement when $(n-c_i)(n-c_i-x_i)=X$ for each $i$. To prove this take $x_i, c_i$ such that $C+X$ is minimized.  For fixed $x_i$, it suffices to consider the case where the $c_i$ are small as possible. The allowable range for each $c_i$ is the intersection of the intervals $[0, n]$ and $\bigg[\frac{2n-x_i-\sqrt{x_i^2+4X}}{2}, \frac{2n-x_i+\sqrt{x_i^2+4X}}{2}\bigg]$.
Hence it suffices to set $c_i$ equal to $\max\bigg\{0, \frac{2n-x_i-\sqrt{x_i^2+4X}}{2}\bigg\}$. Now suppose that some $c_i$ is equal to $0$ with $\frac{2n-x_i-\sqrt{x_i^2+4X}}{2}<0$. We claim that $c_j>0$ for some $j$. Indeed, suppose otherwise. Then $n(n-x_i)\le X$ for each $i$, and summing over $i$ yields $n^3\le 2nX$, or $X\ge\frac{n^2}{2}$, which is false in this case. So $c_j>0$ for some $j$. But then we can simultaneously decrease $x_i$ and increase $x_j$ at the same rate until $\frac{2n-x_i-\sqrt{x_i^2+4X}}{2}=0$. At this point, $c_i=0$ still satisfies the required condition, but $c_j$ can be replaced with a $c_j'<c_j$ since $x_j$ increased. 

So, it suffices to prove the statement in the case that $c_i=\frac{2n-x_i-\sqrt{x_i^2+4X}}{2}\ge 0$ for each $i$. Then
\begin{align*}
C+X&=\sum_{i=1}^n\frac{2n-x_i-\sqrt{x_i^2+4X}}{2}+\sum_{i=1}^nx_i
\\ &=\sum_{i=1}^n (n+\frac{x_i-\sqrt{x_i^2+4X}}{2})
\end{align*}
 We focus on minimizing this last expression. In addition to $x_i\ge 0$, the additional condition $c_i\ge 0$ implies that $\frac{2n-x_i-\sqrt{x_i^2+4X}}{2}\ge 0$ for each $i$, so that $x_i\le n-\frac{X}{n}$ for each $i$. 

Because the function $f(x)=\sqrt{x^2+c}$ is convex for $x\ge 0$ when $c\ge 0$, it follows that the expression above is minmized when all but one of the $x_i$ are equal to one of the boundaries of the interval $[0, n-\frac{X}{n}]$. Let $A=\lfloor\frac{X}{n-\frac{X}{n}}\rfloor$, so that there are $A$ values $i$ with $x_i=A$. Then:
\begin{align*}
\sum_{i=1}^n (n+\frac{x_i}{2}-\frac{\sqrt{x_i^2+4X}}{2})&\ge n^2+\frac{X}{2}-\frac{1}{2}(A+1)\sqrt{(n-\frac{X}{n})^2+4X}-(n-A)\sqrt{X}
\\ &\ge n^2+\frac{X}{2}-\frac{1}{2}(\frac{X}{n-\frac{X}{n}})(n+\frac{X}{n})-(n-\frac{X}{n-\frac{X}{n}})\sqrt{X}-2n
\\ &= n^2(\frac{t}{2}+1-\frac{t(1+t)}{2(1-t)}-(1-\frac{t}{1-t})\sqrt{t})-2n
\end{align*}
as required.
\end{proof}

\begin{lemma}
Suppose that $c_i, x_i$ are nonnegative integers with $c_i\in [0, n], x_i\in [0, n^2]$ for $1\le i\le n$. Let $C=\sum_{i=1}^nc_i$ and $X=\sum_{i=1}^nx_i$, and suppose that $C\ge \frac{X}{2}$ and $(n-c_i-x_i)(n-c_i)\le X$ for each $i$. Further suppose that $\sum_{i=1}^n (X-(n-c_i-x_i)(n-c_i))\ge \frac{X^2}{2n-\frac{X}{n}}$. Then $C+X\ge (\frac{9-3\sqrt{5}}{4}) n^2-2n$.
\end{lemma}
\begin{proof}
We take two cases based on the size of $X$. Let $\alpha=\frac{9-3\sqrt{5}}{4}$. 

Case 1: $X\ge \frac{2\alpha}{3}n^2$. Then $C+X\ge \frac{3X}{2}\ge \alpha n^2$ as required.

Case 2: $X\le \frac{2\alpha}{3}n^2$. Then consider the indices $i$ for which $c_i>\max\{0, \frac{2n-x_i-\sqrt{x_i+4X}}{2}\}$. For each such index $i$, decrease the $c_i$ until it equals this maximum. Repeat this procedure for all of the $c_i$. Then, repeat the procedure described in Lemma 15 so that each $c_i$ is equal to $\frac{2n-x_i-\sqrt{x_i+4X}}{2}$. Let $C'$ denote the new sum of the $c_i$. Then according to Lemma 15:
\[C'+X\ge n^2(1-\frac{t^2}{1-t})-2n\]
where $t=\frac{X}{n^2}$ as before. Now we claim $C\ge C'+\frac{n^2t^2}{2(2-t)}$. Indeed, note that:
\[\frac{\partial}{\partial c_i}((n-c_i)(n-c_i-x_i))=2c_i+x_i-2n\ge -2n\].
It follows that, in the process of decreasing $\sum_{i=1}^n (n-c_i)(n-x_i-c_i)$ by $\frac{X^2}{2(2n-\frac{X}{n})}$, the sum of the $c_i$ decreases by at least $\frac{X^2}{2n(2n-\frac{X}{n})}=\frac{n^2t^2}{2(2-t)}$ as required. Hence:
\begin{align*}
C+X&=(C'+X)+(C-C')    
\\ &\ge n^2(\frac{t}{2}+1-\frac{t(1+t)}{2(1-t)}-(1-\frac{t}{1-t})\sqrt{t}+\frac{t^2}{2(2-t)})-2n
\\ &\ge  \alpha n^2-2n
\end{align*}
where we here used $t\le \frac{2\alpha}{3}$, and that:
\[f(x)=\frac{x}{2}+1-\frac{x(1+x)}{2(1-x)}-(1-\frac{x}{1-x})\sqrt{x}+\frac{x^2}{2(2-x)}\]
is decreasing on $(0, 0.4)$ with $f(\frac{2\alpha}{3})=\alpha$.
\end{proof}

\begin{lemma}
Suppose that $0<X\le n^2$ squares are marked in a $n\times n$ grid, and in each marked square is written either the number of marked squares in the same row, or the number of marked squares in the same column. Then the sum of the written numbers is greater than $\frac{X^2}{2n-\frac{X}{n}}$. 
\end{lemma}
\begin{proof}
We claim that the sum of the reciprocal of the written numbers is at most $2n-\frac{X}{n}$. Indeed, for a marked square $m_i$, let $c_i, n_i$ denote the number of marked squares in the chosen and not chosen direction of $m_i$ respectively. Then:
\begin{align*}
\sum_{i=1}^X\frac{1}{c_i}&= \sum_{i=1}^X(\frac{1}{c_i}+\frac{1}{n_i})-\sum_{i=1}^X\frac{1}{n_i}
\\ &= 2n-\sum_{i=1}^X\frac{1}{n_i}
\\ &\le 2n-\frac{X}{n}
\end{align*}
Then the result follows from Cauchy-Schwarz, since:
\[(\sum_{i=1}^X\frac{1}{c_i})(\sum_{i=1}^Xc_i)\ge X^2\].
\end{proof}

Now we proceed to the proof of the lower bound. Suppose that there is a configuration of $2$-rooks which covers an $H_{n, 3}$ situated at $1\le x, y, z\le n$ in coordinate space. We may orient this configuration so that at least $\frac{1}{3}$ of the rooks cover in the $x$ and $y$ directions. We call these rooks \textit{cross rooks,} and all other rooks \textit{axis rooks}. For each $i, 1\le i\le n$,we denote by $x_i$ the number of axis rooks which lie in the plane $z=i$, and $c_i$ the number of cross rooks which lie in this plane. Let $C=\sum_{i=1}^nc_i, X=\sum_{i=1}^nx_i$, so that $C\ge \frac{X}{2}$. Note that we may assume $c_i\le n$ since $n$ cross rooks can already cover a plane. We claim that, in the $i$th plane, at most $n^2-(n-c_i)(n-c_i-x_i)$ points are covered by rooks in that plane. Indeed, suppose that $h_i$ of the $x_i$ axis points choose to cover a row in $z=i$, and $v_i$ choose to cover a column, so that $v_i+h_i=x_i$. Then at most $c_i+v_i$ rows are covered by the rooks in plane $i$, and at most $c_i+h_i$ columns are covered. Hence in total, at most:
\[n^2-(n-c_i-h_i)(n-c_i-v_i)=n^2-(n-c_i)(n-c_i-x_i)-h_iv_i\le n^2-(n-c_i)(n-c_i-x_i)\]
points in plane $i$ are covered by rooks in that plane as required. It follows that the remaining points in plane $i$ must by covered by axis points from other planes, so that in particular $X\ge (n-c_i)(n-c_i-x_i)$ for every $i$. Furthermore, due to the combination of lemmas 15, 16 it follows that:
\[\sum_{i=1}^n (X-(n-c_i-x_i)(n-c_i))\ge \frac{X^2}{2n-\frac{X}{n}}\]
Therefore, the $x_i, c_i$ satisfy the conditions given in Lemma 16, so it follows that the total number of rooks used is:
\[C+X\ge \alpha n^2-2n\]
Hence $a_{3, 2}\ge \alpha$ as required.
\end{proof}
\section{Bounds for Packing}
In this section we prove that $b_{k,\ell}=\frac{k}{\ell}$ and $c_{k,\ell}=\frac{\binom{k}{2}}{\binom{\ell}{2}}$ for certain special values of $k$ and $\ell$. We begin by demonstrating that this equality holds when $\ell$ divides $k$.

\begin{thm}\label{firstStep}
For positive integers $k, t$, we have $b_{kt, t}=k$. 
\end{thm}
\begin{proof}
By Theorem \ref{Singleton}, it follows that $b_{kt,t}\le k$. Therefore, it suffices to demonstrate that $b_{kt,t}\ge k$. We prove this by demonstrating that $b_{n,kt,t}\ge kn^{k(t-1)}(n-1)^{k-1}$ through explicit construction.

Consider points of the form $(x_1,\ldots,x_t,x_{t+1},\ldots,x_{2t},\ldots, x_{kt})$ with $0\le x_i\le n-1$ for $1\le i\le kt$. Define the $L_j$ block of points as the set of points that satisfy 
\[\sum_{i=0}^{t-1}x_{jt-i}\equiv 0\mod n\] and satisfy for $m\in \{1,\ldots, \ell\}$ and $m\neq j$,
\[\sum_{i=0}^{t-1}x_{mt-i}\nequiv 0\mod n.\] For each point in the $L_j$ block, place an $\ell$-rook that attacks in the direction of the $((j-1)t+1)^{th}$ coordinate to the ${jt}^{th}$ coordinate. Note that $|L_j|=n^{t-1}(n^{t-1}(n-1))^{k-1}=n^{k(t-1)}(n-1)^{k-1}$ and thus taking the union of these rooks for $1\le j\le k$ it follows that \[\left|\bigcup_{j=1}^{k}L_j\right|=kn^{k(t-1)}(n-1)^{k-1}.\]
Now we demonstrate that no point attacks another in the above constructions. Suppose for the sake of contradiction that $R_1$ attacks $R_2$ with $R_1\in L_i$ and $R_2\in L_j$. If $i\neq j$ then note that $R_1$ and $R_2$ differ in at least one coordinate in $x_{(i-1)t+1},\ldots,x_{it}$ and at least one coordinate in $x_{(j-1)t+1},\ldots,x_{jt}$. Since attacking rooks differ by at most $1$ coordinate, such rooks $R_1$ and $R_2$ do not exist. Otherwise $R_1$ and $R_2$ both lie in $L_i$. If these points differ in the coordinates $x_{(i-1)t+1},\ldots,x_{it}$ then they differ in at least two positions and therefore they cannot attack each other. Otherwise $R_1$ and $R_2$ differ outside of the coordinates $x_{(i-1)t+1},\ldots,x_{it}$, and since $R_1$ and $R_2$ attack in these coordinates, $R_1$ does not attack $R_2$ and the result follows.
\end{proof}
We can now establish a crude lower bound for $b_{k,\ell}$.
\begin{corollary}
For positive integers $k, \ell$, $b_{k, \ell}\ge \lfloor\frac{k}{\ell}\rfloor$. 
\end{corollary}
\begin{proof}
Note that $nb_{n,k,\ell}\le b_{n,k+1,\ell}$ as we can stack $n$ constructions of $b_{n,k,\ell}$ in the $(k+1)^{st}$ dimension. Therefore it follows that $b_{k,\ell}\le b_{k+1,\ell}$ and that $b_{k,\ell}\ge b_{\ell\lfloor\frac{k}{\ell}\rfloor,\ell}=\lfloor\frac{k}{\ell}\rfloor$ where we have used Theorem \ref{firstStep} in the final step. 
\end{proof}
The last bound we establish for $b_{k,\ell}$ is that in fact $b_{k,2}=\frac{k}{2}$ for all integers $k\ge 2$.
\begin{thm}
For integers $k\ge 2$, we have $b_k=\frac{k}{2}$. 
\end{thm}
\begin{proof}
We will provide an inductive construction based on constructions in Section 2. In particular, we show the following.

Claim: For every integer $k\ge 2$, there exists a nonnegative constant $c_k$ such that  $b_{n, k, 2}\ge \frac{k}{2}n^{k-1}-c_kn^{k-2}$. 

For the base case $k=2$, we may take $c_2=0$, and place $2$-rooks along the main diagonal of $H_{n, 2}$ which is a size-$n$ square. Suppose the claim holds for all $j\le k-1$ and that $(k-1)!^2$ divides $n$. Then there exists some set $S$ of $\frac{k-1}{2}n^{k-2}-c_{k-1}n^{k-3}$ $2$-rooks that tiles $H_{n, k-1}$. 

We now describe a way to pack \[\frac{(k-1)^k(\frac{n}{k-1})!}{(\frac{n}{k-1}-k+1)!}\] labeled $1$-rooks into a $H_{n, k-1}$ so that no two $1$-rooks of the same label attack each other. For this, we first group some of the $n^{k-1}$ points in the hypercube into \[\frac{(\frac{n}{k-1})!}{(\frac{n}{k-1}-k+1)!}\] buckets of size $(k-1)^{k-1}$. We do this by sending the point $(x_1, \ldots, x_{k-1})$ to a bucket labeled $(\lfloor\frac{x_1}{k-1}\rfloor, \ldots, \lfloor \frac{x_k}{k-1}\rfloor)$ if and only if the $\lfloor\frac{x_i}{k-1}\rfloor$ are distinct. 

Notice that the points in each bucket form a $H_{k-1, k-1}$. Within each bucket, we partition the points into $k-1$ parts of the form $\sum_{i=1}^{k-1}x_i\equiv j\mod k-1$, each of which has $(k-1)^{k-2}$ points. We then label each point in the $j$th part of such a partition with the label $\lfloor \frac{x_j}{k-1}\rfloor$.

When this is done, there are \[\frac{(k-1)^k(\frac{n}{k-1})!}{n(\frac{n}{k-1}-k+1)!}\] points of label $i$ for each $i\in \{1, \ldots, \frac{n}{k-1}\}$. All points of label $i$ have $\lfloor \frac{x_j}{k-1}\rfloor=i$ for exactly one index $j$. Therefore, assigning all points of label $i$ to attack in the direction corresponding to this direction yields a packing of $\frac{(k-1)^k(\frac{n}{k-1})!}{(\frac{n}{k-1}-k+1)!}$ labeled $1$-rooks into a $H_{n, k-1}$ so that no two $1$-rooks of the same label attack each other, as required.

Now we combine this partition $P$ with the set $S$. For $1\le x_k\le \frac{n}{k-1}$, we let the last coordinate act as the label for $P$ and have all of these rooks attack in the direction of the last coordinate in addition to their normal direction. For $\frac{n}{k-1}+1\le x_k\le n$, we fill each layer corresponding to a fixed $x_k$ according to $S$. The number of points in the construction at this point is at least \[\frac{(k-1)^k(\frac{n}{k-1})!}{(\frac{n}{k-1}-k+1)!}+(\frac{k-2}{k-1}n)(\frac{k-1}{2}n^{k-1}-c_kn^{k-2})\ge \frac{k}{2}n^{k-1}-(\frac{k-2}{k-1}c_k+\frac{k(k-1)^2}{2})n^{k-2}.\]
Here we have used the estimate that \[\frac{(\frac{n}{k-1})!}{(\frac{n}{k-1}-k+1)!}\ge (\frac{n}{k-1})^{k-1}-\frac{k(k-1)}{2}(\frac{n}{k-1})^{k-2}.\]
At this point, the only pairs of $2$-rooks that attack other $2$-rooks are the rooks from $P$ that attack the rooks from $S$. But there are at most $\frac{k-1}{2}n^{k-2}$ points in $S$ and each point in $S$ lead to at most $1$ offending point in $P$. Therefore we may simply remove these points to obtain a configuration of at least
$\frac{k}{2}n^{k-1}-(\frac{k-2}{k-1}c_k+\frac{k(k-1)^2}{2}+\frac{k-1}{2})n^{k-2}$
$2$-rooks, none of which attack each other.

It follows that $b_{k, 2}=\lim_{t\to\infty}b_{(k-1)!^2t, k, 2}=\frac{k}{2}$, as desired.
\end{proof}

Transitioning, we now determine $c_{k,2}$ and $c_{k,k}$. The second constant is known implicitly in the literature, but the proof is included in the following theorem. 
\begin{thm}
For all positive integers $k$, $c_{k,k}=1$ and for $k\ge 2$, $c_{k,2}=\binom{k}{2}$.
\end{thm}
\begin{proof}
We begin by proving that $c_{k,k}=1$. By Theorem \ref{Singleton}, $c_{k,k}\le 1$. The construction that $c_{k,k}\ge1$ is the exactly the one given in Theorem \ref{cExists} as this demonstrates $c_{p,k,k}=p^{k-2}$ for primes $p>k$. The result follows.

For the second part of the proof, note that $c_{k,2}\le \binom{k}{2}$ by Theorem \ref{Singleton}. Therefore, it suffices to demonstrate that $c_{k,2}\ge \binom{k}{2}$. To demonstrate this we prove that $c_{n,k,2}\ge \binom{k}{2}(n-2\binom{k}{2})^{k-2}$ for $n>2\binom{k}{2}$. In particular, for $i<j$, let $A_{i,j}$ be the set of points with $i^{th}$ and $j^{th}$ coordinates being $2i-2+(j-1)(j-2)$ and $2i-1+(j-1)(j-2)$ respectively, and other coordinates varying in the range $[k(k-1)+1,n-1]$. Note that $2i-2+(j-1)(j-2)$ and $2i-1+(j-1)(j-2)$ lie between $[0,k(k-1)]$ and take each value in this range exactly one. Now for each point in $A_{i,j}$, orient the corresponding $2$-rook to attack in the direction of the $i^{th}$ and $j^{th}$ axes. No two points within a set attack each other as they differ outside the $i^{th}$ and $j^{th}$ coordinates and any pair of rooks from differing sets differ in at least $3$ coordinates and therefore cannot attack each other. Therefore, the result follows by taking all $A_{i,j}$ where the $2$-rooks in $A_{i,j}$ are directed to attack along the  $i^{th}$ and ${j}^{th}$ dimension.
\end{proof}

\section{Conclusion and Open Questions}
There are several natural questions and conjectures regarding the values of $a_{k,\ell}$, $b_{k,\ell}$, and $c_{k,\ell}$. The most 
surprising open question is the following.
\begin{conj}
For integers $k\ge 2$, $a_{k,k}=\frac{1}{k-1}$.
\end{conj}
Note that the above conjecture is known when $k$ is one more than a power of a prime \cite{blokhuis1984more}, and the construction is essentially that of perfect codes. This construction implies that the first open case of this conjecture is $a_{7,7}$. The most natural case when $k\neq \ell$ and that is not covered by our result is $a_{3,2}$.
\begin{question}
What is the exact value of $a_{3,2}$?
\end{question}

We end on pair of conjectures based on the results of the previous section, and in contrast to $a_{k,\ell}$ we conjecture exact values for $b_{k,\ell}$ and $c_{k,\ell}$ that are upper bounds from Theorem \ref{Singleton}.

\begin{conj}
For positive integers $k\ge \ell$, $b_{k,\ell}=\frac{k}{\ell}$.
\end{conj}

\begin{conj}
For positive integers $k\ge \ell\ge 2$, $c_{k,\ell}=\frac{\binom{k}{2}}{\binom{\ell}{2}}$.
\end{conj}
\section{Acknowledgements}
This research was conducted at the University of Minnesota Duluth REU and was supported by NSF grant 1659047. We would like to thank Joe Gallian, Colin Defant, and Ben Gunby for reading over the manuscript and providing valuable comments. We would especially like to thank Ben Gunby for finding a critical error in an earlier version of the paper.
\bibliographystyle{plain}
\bibliography{bibliograph.bib}
\end{document}